\documentclass[11pt]{article}

\usepackage{epsfig}
\usepackage{graphicx}
\usepackage{amsbsy}
\usepackage{amsmath}
\usepackage{amsfonts}
\usepackage{amssymb}
\usepackage{textcomp}
\usepackage{hyperref}
\usepackage{aliascnt}

\newcommand{\mcm}[3]{\newcommand{#1}[#2]{{\ensuremath{#3}}}} 

\mcm{\tuple}{1}{\langle #1 \rangle}
\mcm{\name}{1}{\ulcorner #1 \urcorner}
\mcm{\Nbb}{0}{\mathbb{N}}
\mcm{\Zbb}{0}{\mathbb{Z}}
\mcm{\Rbb}{0}{\mathbb{R}}
\mcm{\Cbb}{0}{\mathbb{C}}
\mcm{\Qbb}{0}{\mathbb{Q}}
\mcm{\Acal}{0}{\cal A}
\mcm{\Bcal}{0}{\cal B}
\mcm{\Ccal}{0}{\cal C}
\mcm{\Dcal}{0}{\cal D}
\mcm{\Ecal}{0}{\cal E}
\mcm{\Fcal}{0}{\cal F}
\mcm{\Gcal}{0}{\cal G}
\mcm{\Hcal}{0}{\cal H}
\mcm{\Ical}{0}{\cal I}
\mcm{\Jcal}{0}{\cal J}
\mcm{\Kcal}{0}{\cal K}
\mcm{\Lcal}{0}{\cal L}
\mcm{\Mcal}{0}{\cal M}
\mcm{\Ncal}{0}{\cal N}
\mcm{\Ocal}{0}{{\cal O}}
\mcm{\Pcal}{0}{{\cal P}}
\mcm{\Qcal}{0}{{\cal Q}}
\mcm{\Rcal}{0}{{\cal R}}
\mcm{\Scal}{0}{{\cal S}}
\mcm{\Tcal}{0}{{\cal T}}
\mcm{\Ucal}{0}{{\cal U}}
\mcm{\Vcal}{0}{{\cal V}}
\mcm{\Wcal}{0}{{\cal W}}
\mcm{\Xcal}{0}{{\cal X}}
\mcm{\Ycal}{0}{{\cal Y}}
\mcm{\Zcal}{0}{{\cal Z}}
\mcm{\Mfrak}{0}{\mathfrak M}

\mcm{\restric}{0}{\upharpoonright}
\mcm{\upset}{0}{\uparrow}
\mcm{\onto}{0}{\twoheadrightarrow}
\mcm{\smallNbb}{0}{{\small \mathbb{N}}}
\DeclareMathOperator{\preop}{op}
\mcm{\op}{0}{^{\preop}}

\newcommand{\se}{\subseteq}
{\begin{array}{c}
\setlength{\unitlength}{1em}}%
{\end{array}}

\usepackage{amsthm}

\newcommand{\theoremize}[2]{\newaliascnt{#1}{thm} \newtheorem{#1}[#1]{#2} \aliascntresetthe{#1}}

\theoremstyle{plain}
\newtheorem{thm}{Theorem}[section]
\theoremize{lem}{Lemma}
\theoremize{skolem}{Skolem}
\theoremize{fact}{Fact}
\theoremize{sublem}{Sublemma}
\theoremize{claim}{Claim}
\theoremize{obs}{Observation}
\theoremize{prop}{Proposition}
\theoremize{cor}{Corollary}
\theoremize{que}{Question}
\theoremize{oque}{Open Question}
\theoremize{con}{Conjecture}

\theoremstyle{definition}
\theoremize{dfn}{Definition}
\theoremize{rem}{Remark}
\theoremize{eg}{Example}
\theoremize{exercise}{Exercise}
\theoremstyle{plain}

\usepackage{verbatim}
\usepackage{enumerate}
\usepackage[all]{xy}

\usepackage{subfig}

\title{Characterising graphs with no subdivision of a wheel of bounded diameter
}

\author{Johannes Carmesin
\medskip 
\\
  {University of Birmingham}
}
\newcommand{\sm}{\setminus}

\DeclareMathOperator{\expl}{Ex_r}

\mcm{\Fbb}{0}{\mathbb{F}}

\usepackage{xr}
\begin{document}

\maketitle

\begin{abstract}
We prove that a graph has an $r$-bounded subdivision of a wheel if and only if it 
does not have a graph-decomposition of locality $r$ and width at most two. 
\end{abstract}

\section{Introduction}

In \cite{loc2sepr} we introduced local 
2-separators of graphs and proved a local version of the 2-separator theorem. Here, we give an 
application of that theorem. 

An important class of graphs is the class of series-parallel graphs. A well-known fact about this 
class is that a graph has a tree-decomposition of width at most two if and only if it has no 
$K_4$-subdivision. We prove an analogue of this fact in our new context 
of 
local 
separators. 

Examples of graphs of diameter $r$ are $r$-bounded subdivisions, see \autoref{fig1} for an 
example and \autoref{sec2xz} for details. The 
main result of this paper is 
the 
following.

\begin{thm}\label{intro_series-para}
 Let $G$ be a graph and $r\in \Nbb\cup\{\infty\}$ be a parameter. Then precisely one of 
the 
following holds.\vspace{-.1cm}
\begin{enumerate}
 \item $G$ has an $r$-bounded subdivision of a wheel;
 \item $G$ has a graph-decomposition of locality $r$ and width at 
most two. 
\end{enumerate}

\end{thm}

   \begin{figure} [htpb]   
\begin{center}
   	  \includegraphics[height=4cm]{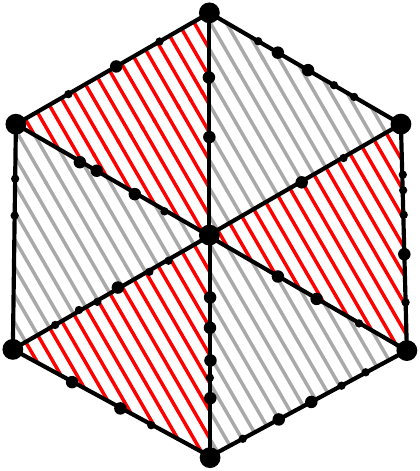}
   	  \caption{An example of a subdivision of a wheel that is $r$-bounded. Here all the 
shaded triangularly 
shaped faces have length at most $r$. 
}\label{fig1}
\end{center}
   \end{figure}

  {\bf Related results.} 
  Shallow minors were introduced in 1994 by Plotkin, Rao and Smith \cite{plotkin1994shallow}, 
  and bounded subdivisions are a particular type of shallow minors. 
Recently shallow minors received attention in the context of the systematic analysis of Sparsity, 
see in particular the book of Ossana de Mendez and Ne{\v{s}}et{\v{r}}il 
\cite{nevsetvril2012sparsity} and 
\cite{{eppstein2013grid},{gajarsky2018first},{kriesell2017unique},{muzi2017being}} for some recent 
work on shallow minors. While research in this area so far has focused on exploring the asymptotic 
relationships between various classes 
of sparse graphs, to my knowledge, 
\autoref{intro_series-para} is the first genuine example of an explicit characterisation by 
excluded shallow minors or bounded subdivisions.  
  
\vspace{0.3cm}  
  
The grid theorem \cite{{excludedGrid},{robertson1986graph}} says 
that a graph either has a large grid as a minor or else a tree-decomposition of bounded width. In 
spirit this is a similar statement to \autoref{intro_series-para}: the absence of a specific 
substructure is characterised by the existence of a global graph-decomposition. However, there is a 
difference in the quantification. Indeed, \autoref{intro_series-para} is \emph{exact} in the sense 
that the two conditions in there are \emph{mutually exclusive}, unlike for the grid 
theorem (even in the most recent version of \cite{chekuri2016polynomial}). 

Our research suggests the question whether there is an analogue of 
the grid theorem for shallow minors as follows. 

\begin{oque}
Given a parameter $r$, can you characterise the graphs with graph-decompositions of bounded width 
and locality at least $r$?
\end{oque}

\section{Bounded subdivisions}\label{sec2xz}

We assume that the reader is familiar with \cite{loc2sepr}; in particular with the definition of 
local cutvertices and local 2-separators. 
A connected graph is \emph{$r$-locally 2-connected} if it does not have an $r$-local cutvertex and 
it has a cycle of length at most $r$ (in particular such graphs have at least three vertices).
So there are no $r$-locally $2$-connected graphs for $r<3$. 
A graph is \emph{$r$-locally $2$-connected} if all its components are $r$-locally $2$-connected.

A connected $r$-locally 2-connected graph is \emph{$r$-locally 3-connected} if it does not have an 
$r$-local $2$-separator 
and it has 
at least four vertices.
A graph is \emph{$r$-locally $3$-connected} if it $r$-locally 2-connected and all its components 
are $r$-locally $3$-connected.

\begin{lem}\label{gen_loc_con}
 A 3-connected graph $G$ whose cycles of length at most $r$ generate all its cycles is 
$r$-locally $3$-connected.
 \end{lem}

 \begin{proof}
 By definition $3$-connected graphs have at least four vertices. 
 
Suppose for a contradiction that the graph $G$ has an $r$-local cutvertex. Call it $v$.  
Let $P$ be a path of $G$ joining two neighbours of $v$ in different components of the punctured 
ball $B_{r/2}(v)-v$. Then $P+v$ is a cycle. By assumption the cycle $P+v$ is generated by cycles of 
length at most $r$. Pick a components of $B_{r/2}(v)-v$ containing an endvertex of the path $P$. 
Denote 
it by $K$. 
Let $W$ be the set of edges incident with the vertex $v$ to the component $K$. 
As the cycle $P+v$ intersects the edge set $W$ oddly (in fact just once), there needs to be a 
generating cycle $o$ such 
that it intersects the edge set $W$ oddly; and so just once as it no more than two edges incident 
with the vertex $v$. As $o$ is a cycle of the ball $B_{r/2}(v)$, we conclude that $o-v$ is a subset 
of 
the component $K$. So $o$ intersects the edge set $W$ in two edges. This is a contradiction.
Thus $G$ has no $r$-local cutvertex.

The proof that $G$ has no $r$-local 2-separator is analogous. 
 \end{proof}

A \emph{wheel} is a graph obtained from a cycle by adding a single vertex and connecting it to all 
vertices of the cycle.  
A subdivision of a wheel is \emph{$r$-local} if its cycles of length at most $r$ generate all 
its cycles. An \emph{$r$-weighted wheel} is a graph isomorphic to a 
wheel with integer-valued length assigned to its edges such that if one replaced each edge by a 
path of its length, then the resulting graph is an $r$-local subdivision of a wheel. 

\begin{lem}\label{are-3-con}
$r$-weighted wheels are $r$-locally 3-connected. 
\end{lem}

\begin{proof}
 Wheels are $3$-connected. Hence this follows from \autoref{gen_loc_con}. 
\end{proof}

\begin{rem}
 All graphs in this paper are weighted graphs; that is, graph with length assigned to their edges. 
All edge-length are positive integers. For most of our lemmas it does not make a difference that we 
consider weighted graphs instead of usual graphs. In the few 
places, where it matters, we point that out explicitly. It is essential in this paper that we work 
with weighted graphs as the torsos of graph-decompositions are weighted graphs. 

Moreover, all graphs considered in the paper are simple; that is, they do not have loops or 
parallel edges. 
\end{rem}

\begin{lem}\label{bd_diam}
 Any $r$-local subdivision of a wheel has diameter at most $r$. 
\end{lem}

\begin{proof}
 We denote the center of the wheel by $c$. The \emph{rim} is the unique cycle of a (subdivision of 
a) wheel that does not contain the vertex $c$. We distinguish two cases.

 {\bf Case 1:} the rim has length more than $r$. By assumption, every vertex is in a cycle of 
length at most $r$. This cycle must contain the center $c$. So every vertex has distance at most 
$r/2$ from the center. So the diameter is at most $r$. 

{\bf Case 2:} the rim has length at most $r$. 
Let $x$ and $y$ be two arbitrary vertices. By assumption they are each contained in a cycle of 
length at most $r$. 
If both these cycles contain the center, then both these vertices have distance at most $r/2$ from 
the center, so distance at most $r$ from another.
Otherwise one of these cycles must be the rim. The other cycle must contain a vertex from the rim. 
So $x$ and $y$ both have distance at most $r/2$ from that vertex, and so distance at most $r$ from 
another.
\end{proof}

We refer to the 
triangles of a wheel containing the center of the wheel as the \emph{pieces} of the 
wheel (these are all the triangles of the wheel except in the case of $K_4$).
 A \emph{piece} 
of a subdivision of a wheel is a cycle obtained by subdividing a piece. 

In our definition of when a subdivision of a wheel is $r$-local, we were flexible about the 
structure of the generating set of cycles. As most cycles in wheels use the 
center, we can -- in fact -- be more explicit as follows.
We say that a subdivision of a wheel is \emph{$r$-bounded} if all its pieces have length at most 
$r$.  Our aim is to prove the following.

\begin{lem}\label{make_bounded}
 Every $r$-local subdivision of a wheel contains a subdivision of a wheel that is $r$-bounded. 
\end{lem}

First we do some preparation. 
We say that a subdivision of a wheel has an \emph{$r$-explicit} generating set if either all its 
pieces have length at most $r$ or else all its pieces except for one have length at most $r$ 
and the rim (that is, the unique cycle not containing the center) has length at most $r$.
\begin{eg}
 Every subdivision of a wheel with an $r$-explicit generating set is $r$-local. 
\end{eg}

\begin{lem}\label{bounded2}
 Every subdivision of a wheel with an $r$-explicit generating set contains a subdivision of a wheel 
that is $r$-bounded. 
\end{lem}

\begin{proof}
 It suffices to show that every subdivision of a wheel where all but one of the pieces have length 
at most $r$ and the rim has length at most $r$, contains an $r$-bounded subdivision of a wheel. 
In fact we will find an $r$-bounded subdivision of the 3-wheel $K_4$ as follows. 
By assumption there are two adjacent pieces -- that is, two pieces sharing at least one edge -- of 
length at most $r$. Each of these pieces shares at least one edge with the rim.
Let $H$ be the union of the rim with these two adjacent pieces. As not all edges of the rim are in 
these two adjacent pieces, the graph $H$ is a subdivision of $K_4$. There is a unique vertex in the 
intersection of the old rim and the two old adjacent pieces. We see this graph $H=K_4$ as a 
$3$-wheel with that vertex as the center. Then the pieces are precisely the two adjacent old pieces 
and the old rim. So the graph $H$ is an $r$-bounded subdivision of a $3$-wheel. 
\end{proof}

\begin{lem}\label{make_explicit}
 Every $r$-local subdivision of a wheel contains a subdivision of a wheel that has an 
$r$-explicit generating set. 
\end{lem}

\begin{proof}
By suppressing vertices of degree two if necessary, it suffices 
to prove the lemma for $r$-weighted wheels. Let $W$ be an $r$-weighted wheel. 
We prove this lemma by induction on $W$. If the wheel $W$ is the graph $K_4$, all its cycles are 
pieces 
or the rim. Hence the claim is trivially true.
Now assume that $W$ is a wheel with at least four vertices on the rim. 

{\bf Case 1:} all geodesic cycles of $W$ of length at most 
$r$ are pieces or the rim. As every cycle in the generating set has length at most $r$, it is 
generated by geodesic cycles of length at most $r$. 
Hence the pieces and the rim together generate all cycles of $W$. So either all pieces of 
$W$ must have length at most $r$, or else the rim has length at most $r$ and all but at most 
one piece has length at most $r$. Thus the wheel $W$ has an 
$r$-explicit generating set.

{\bf Case 2:} not Case 1; that is, there is a geodesic cycle $o$ of $W$ of length at most 
$r$ that is not the rim or a piece. As the rim is the only cycle that does not contain the 
center of the wheel $W$, the cycle $o$ must contain the center. 
As the cycle $o$ is not a piece, it must contain at least one chord, denote it by $x$. 
Let $W'$ be the subgraph of 
$W$ obtained by deleting the chord $x$ of the cycle $o$. As the wheel $W$ is not the $3$-wheel 
$K_4$, the subgraph $W'$ is a subdivision of a wheel. Let $\Ccal$ be a set of cycles of length at 
most $r$ 
generating all cycles of $W$. 
Let $P$ be a shortest path between the two endvertices of the chord $x$ included in the geodesic 
cycle $o$. 
We obtain $\Ccal'$ from $\Ccal$ by replacing 
in each element of  $\Ccal$ the chord $x$ by the subpath $P$ of $o$. The set $\Ccal'$ is a set of 
closed walks of the graph $W'$ of length at most $r$ that generates all its cycles. Hence the 
subdivision $W'$ is $r$-local. By induction, the graph $W'$ has a subdivision with an 
$r$-explicit generating set. This completes the proof. 
\end{proof}

\begin{proof}[Proof of \autoref{make_bounded}]
 Combine  \autoref{bounded2} and \autoref{make_explicit}.
\end{proof}

\section{Reduction to the locally 3-connected case}

A graph $H$ is called an \emph{$r$-local cut-subgraph} of a graph $G$ if $H$ is obtained from 
$G$ by successively deleting edges and vertices and cutting at $r$-local 1-separators and $r$-local 
2-separators. Here we follow the convention that after cutting at a local 2-separator, we 
immediately replace the torso edges by paths of the same length. 

\begin{eg}\label{eg_careful}
 A cycle of length $\ell\geq r+1$ has a path with the same number of edges as a cut-subgraph but 
not as a subgraph.

If $H$ is a disconnected cut-subgraph of a graph $G$, two of its edges can be copies of the same 
edge of $G$, one of 
which lying on a replacement path for a torso edge. 
\end{eg}

In contrast to \autoref{eg_careful}, if the smaller graph $H$ is sufficiently connected and 
not too big, the cut-subgraph relation is just the usual subgraph relation, as shown in 
the following technical lemma.  

\begin{lem}\label{cut_sub}
 Assume a graph $H$ is an $r$-local cut-subgraph of a graph $G$.
 If the graph $H$ has diameter at most $r$, then it is a subgraph of $G$. 
\end{lem}

\begin{proof}
We prove this by induction on the number of operations required to obtained $H$ from $G$. 
Let $G'$ be the graph obtained from $G$ by performing the first operation. By induction, $H$ is a 
subgraph of $G'$. If $G'$ is a subgraph of $G$ we are fine. Hence we are left with two cases.

{\bf Case 1:} $G'$ is obtained from $G$ by locally cutting a local cutvertex $x$. By the definition 
of 
$r$-local cutting, two slices of the same vertex have distance more than $r$. As the graph $H$ has 
diameter at most $r$, at 
most one slice of the vertex $x$ is in the graph $H$. 
If existent, denote such a slice of $x$ in $H$ by $x'$. Let $G''$ be the subgraph of $G'$ obtained 
by deleting all slices of $x$ different from $x'$ -- and if no slice of $x$ is in $H$, we delete 
all slices of $x$. 
By construction, the graph $G''$ has the graph $H$ as a subgraph. 
The graph $G''$ is equal to the subgraph of $G$ obtained 
by deleting all edges incident with $x$ that are not incident with $x'$.  Thus $H$ is a subgraph of 
the graph $G$. 

{\bf Case 2:} $G'$ is obtained from $G$ by $r$-locally cutting at an $r$-local separator 
$\{x,y\}$.  By {\cite[\autoref*{cut_far}]{{loc2sepr}}}, no two slices of the same vertex of the 
local separator $\{x,y\}$ 
are in the graph $H$.
If one of the vertices $x$ or $y$ has no slice in the graph $H$, we treat this case as Case 1. 
Hence we may assume, and we do assume, that both vertices $x$ and $y$ have slices in the graph 
$H$. Denote these slices by $x'$ and $y'$, respectively.
If these slices come from different components of $\expl(x,y)-x-y$, we treat each slice 
separately as in Case 1.
Hence we may assume, and we do assume, that both slices $x'$ and $y'$ come from the same 
component of $\expl(x,y)-x-y$. Denote that component by $K$. 

Let $P$ be a path of the graph $G$ corresponding to the torso edge $x'y'$ of the graph $G'$.
We denote the path of $G'$ by which the torso edge $x'y'$ is replaced by $P'$. 

If the graph $H$ does not contain any interior vertex of the path $P'$, we treat each slice 
separately as in Case 1. Hence we may assume, and we do assume, that the graph $H$ contains 
an interior vertex of the path $P'$.

\begin{sublem}\label{not_bad}
There does not exist a vertex $z$ of $G$ that has two copies in the graph $H$.
\end{sublem}

\begin{proof}
Suppose for a contradiction there is such a vertex $z$. Then $z$ must have a copy in the path $P'$ 
and a copy in $P$. As the graph $H$ has diameter at most $r$, there is a path $Q$ of length at 
most $r$ between these two copies of $z$ in the graph $H$, and so in the supergraph $G'$. As no 
interior vertex of the path $P'$ has a neighbour outside the path $P'$, the path $Q$ must contain 
one of the endvertices of the path $P'$. Let $v$ be such an endvertex with $Qv\se P'$. Denote by 
$R$ the subpath 
of the path $P$ from which the subpath $Qv$ of $P'$ is cloned from. Then $vQR$ is a walk between 
two different slices of the vertex from which $v$ is cloned from. This path has length at most $r$. 
As two 
slices never have distance at most $r$ by {\cite[\autoref*{cut_far}]{{loc2sepr}}}, we derive at a 
contradiction. So there cannot be a vertex $z$ of $G$ that has two copies in the graph $H$. 
\end{proof}

We obtain the graph $G''$ from the graph $G'$  by deleting all slices of the vertices $x$ and 
$y$ except for $x'$ and $y'$, deleting all torso-edge replacement paths between slices different 
from $x'$ and $y'$ -- and by deleting all vertices of the paths $P$ and $P'$ that are not in the 
graph $H$. 
By construction, the graph $G''$ has the 
graph $H$ as a subgraph.

\begin{sublem}\label{is_sub}
 $G''$ is a subgraph of $G$.
\end{sublem}

\begin{proof}
Label an interior vertex of the path $P$ of $G$ by $P'$ if it has a clone in the graph $H$ that is 
on the path $P'$.
Label it by $P$ if it has a clone in the graph $H$ that is on the path $P$. Otherwise do not give 
it any label. 
We claim that the graph $G''$ is a 
subgraph of the graph $G$.
To see that delete from $G$ all edges incident with $x$ or $y$ that are 
not incident with $x'$ or $y'$, 
delete interior vertices of $P$ without a label, for vertices with label $P'$ delete all its 
incident edges that go to a vertex outside $P'$. By \autoref{not_bad} the resulting subgraph of $G$ 
is equal to $G''$. 
\end{proof}
  
By \autoref{is_sub}, the graph $G$ is a supergraph of the graph $G''$, which in turn is a 
supergraph of $H$.
\end{proof}

\begin{cor}\label{subgraph_for_free}
If a graph $G$ contains a graph $H$ that is an $r$-local subdivision of a wheel as a 
cut-subgraph, then $H$ is a subgraph of $G$.
\end{cor}
\begin{proof}
 By \autoref{bd_diam} $r$-local subdivisions of  wheels have diameter at most $r$. So the 
corollary follows from \autoref{cut_sub}. 
\end{proof}

\begin{thm}\label{series-para}
 Let $G$ be a graph and $r\in \Nbb\cup\{\infty\}$ be a parameter. Then precisely one of the 
following holds.
\begin{enumerate}
 \item $G$ has an $r$-bounded subdivision of a wheel;
 \item $G$ has a graph-decomposition of locality $r$ and adhesion at most two such that all torsos 
are cycles or single edges. 
\end{enumerate}
\end{thm}

\begin{proof}[Proof that \autoref{series-para} implies \autoref{intro_series-para}]
 The only difference between these two theorems is the formulation of the second condition. 
Clearly, any graph-decomposition as in condition 2 of \autoref{series-para} can be refined to yield 
a graph-decomposition as in \autoref{intro_series-para}.
\end{proof}

In the next section we prove the following.

\begin{lem}\label{3con_to_wheel}
 Every $r$-locally $3$-connected graph contains an $r$-local subdivision of a wheel. 
\end{lem}

\begin{proof}[Proof that \autoref{3con_to_wheel} implies \autoref{series-para}.]
Let $G$ be a graph and $r$ be a parameter. Take the $r$-local block-cutvertex-graph decomposition 
of $G$ as in {\cite[\autoref*{block-cut}]{{loc2sepr}}}. Then we take the 
$r$-local 2-separator decomposition {\cite[\autoref*{thm:main_intro}]{{loc2sepr}}} of 
each block that is not a single edge. If all torsos of these decompositions are cycles, we can 
stick these graph-decompositions together to obtain the graph-decomposition in condition 2 of 
\autoref{series-para}.

Hence we may assume, and we do assume, that one of the torsos of these decomposition is 
$r$-locally 3-connected. Call that torso $\beta$. The 2-block $\gamma$ containing $\beta$ is a 
cut-subgraph of $G$. To see that cut all local cutvertices contained in that block and then 
delete all vertices outside that block. And $\beta$ is a cut-subgraph of $\gamma$. To see that 
cut all local 2-separators contained in $\beta$ and then delete all vertices outside $\beta$. 
By  \autoref{3con_to_wheel} the torso $\beta$ has a subgraph $H$ that is an 
 $r$-local 
subdivision of a wheel. To 
summarise, $H$ is a cut-subgraph of $G$. By \autoref{subgraph_for_free}, $H$ is a subgraph of 
$G$; that is, $G$ has an $r$-local subdivision of a wheel. By \autoref{make_bounded}, $G$ has an 
$r$-local subdivision of a wheel.

By \autoref{are-3-con} conditions 1 and 2 in \autoref{series-para} mutually exclusive.
\end{proof}

\section{From local 3-connectivity to a bounded wheel}

This section is dedicated to the proof of \autoref{3con_to_wheel}, which is our last step in the 
proof of the main result stated in the Introduction. 
This proof is subdivided into several subsections. 
In \autoref{sec:dicho}, we prove that any weighted $r$-locally $3$-connected graph 
contains a geodesic cycle of length at most $r$ with at least four edges or else a certain 
$K_4^-$-subgraph. In \autoref{special}, we use this particular $K_4^-$-subgraph to construct a 
bounded wheel as a subgraph. 
Hence it remains to construct a bounded subdivision of a wheel using that geodesic cycle. This is 
done in several steps. 

In \autoref{subsec:fan} we construct a 
bounded fan in any $r$-locally 3-connected graph.
In \autoref{det:bd_wheel}, we give conditions under which we can construct an $r$-local subdivision 
of the $3$-wheel (that is, $K_4$) or the $4$-wheel.
In \autoref{subsec:constr_theta}, 
and 
\autoref{subsec:combine} we show how one can combine 
the results of the previous subsections to give a proof of \autoref{3con_to_wheel}.

\subsection{A dichotomy result}\label{sec:dicho}

A cycle is \emph{geodesic} within a graph $G$ if it contains a shortest path of $G$ between any two 
of its vertices.

We say that a weighted graph $G$ is \emph{triangular} (with parameter $r$) if all its geodesic 
cycles of length at most $r$ are triangles and it contains an edge $e$ that is a shortest path 
between its endvertices that is in at least two triangles of length at most $r$. 

\begin{eg}
A graph that is triangular with any parameter $r\geq 1$ is chordal.
If all edges have length one and the graph has a $K_4^-$ subgraph, the converse is also true.
\end{eg}

\begin{thm}\label{dichotomy}\label{nice_geo}
 Every $r$-locally $3$-connected (weighted) graph has geodesic cycle of length at most $r$ with 
at least four edges or it is triangular with parameter $r$. 
\end{thm}

First we do some preparation. 

\begin{lem}\label{K4minus}
Let $G$ be a graph with a $K_4^-$ subgraph such that its two triangles have 
length at most $r$. Then $G$ contains a cycle of length at most $r$ with at least 
four edges -- or it contains an edge that is a shortest path 
between its endvertices that is in at least two triangles of length at most $r$. 
\end{lem}

\begin{proof}
Let $e$ be the unique edge of the $K_4^-$-subgraph that is in both its triangles.
Let $P$ be a shortest path between the endvertices of the edge $e$. 
If the path $P$ is equal to the edge $e$, we 
are done. Thus we may assume, and we do assume, that the path $P$ contains at least two edges.
If the path $P$ contains at least three edges, consider the cycle $P+e$. This cycle has length at 
most $r$ as it can be obtained from a triangle of $K_4^-$ of length at most $r$ by replacing 
its path of two edges between the endvertices of $e$ by the (shortest) path $P$. 
Hence we may assume, and we do assume, that the path $P$ consists of precisely two edges. Thus 
there is one of the two triangles of $K_4^-$ of length at most $r$ that does not contain the 
middle vertex of the path $P$. We obtain a cycle with four edges from that triangle by replacing 
the edge $e$ by the path $P$. As $P$ is a shortest path, this $4$-cycle has length at most $r$.  
\end{proof}

\begin{lem}\label{simple_lemma}
 Every $r$-locally $3$-connected graph $G$ contains a cycle of length at most $r$ with at least 
four edges -- or it is triangular. 
\end{lem}

\begin{proof}
 As $G$ is $r$-locally $2$-connected, it includes a cycle of length at most $r$. 
 If $r\leq 3$, the cycle $o$ is a triangle. Otherwise if $o$ has at 
least four edges, we are done. So we may assume, and we do assume, that $o$ is a triangle. As $G$ 
is $r$-locally $3$-connected, the component of $o$ has another vertex. Let $v$ be a vertex of $o$ 
that has a neighbour outside $o$. Denote this neighbour by $w$. Consider the punctured ball 
$B_{r/2}(v)-v$. In there, there is a path $P$ from $w$ to some neighbour of $v$ on the triangle 
$o$. 
Then $P+v$ is a cycle. By {\cite[\autoref*{cycle_gen}]{{loc2sepr}}} and 
{\cite[\autoref*{gen}]{{loc2sepr}}}, the cycle $P+v$ is generated by cycles of length at most 
$r$. Similarly as above, we may assume, and we do assume, that all of them are triangles. 
As the cycle $P+v$ contains precisely one edge incident with the vertex $v$ on $o$, one of the 
generating triangles must contain an odd number of edges incident with $v$ on the cycle $o$.
Thus there is a generating triangle $o'$ sharing precisely one edge with the triangle $o$. 
Hence the two triangles $o$ and $o'$ form a $K_4^-$ subgraph. 

If $G$ has a cycle of length at most $r$ that is not a triangle, we are done.
Otherwise, by  \autoref{K4minus} the graph $G$ is triangular.  
\end{proof}

\begin{proof}[Proof of \autoref{dichotomy}.]
By \autoref{simple_lemma}, we may assume, and we do assume, that the graph $G$ has a cycle 
$o$ of length at most $r$ with at least four edges.

Pick a cycle $o$ of minimal length amongst all cycles with at least four edges. 
If $o$ is geodesic, we are done. So suppose that the cycle $o$ is 
not geodesic. Then there
 is a shortcut between two of its vertices; that is a 
path joining two vertices of the cycle $o$ whose length is strictly shorter than the distance 
between these vertices on the cycle $o$. 
Such a shortcut cuts the cycle $o$ into two cycles, each of strictly 
smaller length. 
By minimality of $o$, both these new cycles must be triangles. In particular, the shortcut must be 
a single edge. 
Denote that edge by $e$. To summarise, we have found an edge $e$ that is a shortest path between 
its endvertices that is contained in two triangles of length at most $r$. 

If all geodesic cycles of length at most $r$ of $G$ are triangles, then $G$ is triangular.
Otherwise $G$ has a geodesic cycle of length at most $r$ that is not a triangle, which is the 
other outcome of the theorem. 
This completes the proof. 
\end{proof}

\subsection{The triangular case}\label{special}

In this subsection we prove the following special case of \autoref{3con_to_wheel}. 

\begin{lem}\label{3con_to_wheel_special}
Let $G$ be a (weighted) $r$-locally $3$-connected graph that is triangular with parameter 
$r$.
Then $G$ contains an $r$-weighted wheel as a subgraph. 
\end{lem}

\begin{proof}
As $G$ is triangular, it contains an edge $e$ that is a shortest path between its endvertices, and 
there are two triangles of length at most $r$ containing $e$. 
Denote the two endvertices of the edge $e$ by $v$ and $w$. Denote the two vertices on the triangles 
not incident with the edge $e$ by $x$ and $y$. To summarise, the vertices $x$, $y$, $v$ and $w$ 
span a $K_4^-$-subgraph.

Now we construct the following auxiliary graph. Its vertex set consists of the edges incident 
with the vertices $v$ or $w$ 
different from the edge $e$. Two such edges $e_1$ and $e_2$ are adjacent in this auxiliary graph if 
there is a geodesic cycle of $G$ of length at 
most $r$ containing the edges $e_1$ and $e_2$. We denote this auxiliary graph by $H$. 

\begin{sublem}\label{sublem21}
 If the graph $G$ is $r$-locally $3$-connected, then in the graph $H$ there is a path from the 
vertex $vx$ of $H$ to the vertex $vy$. 
\end{sublem}

\begin{proof}
By assumption, the punctured explorer-neighbourhood $\expl(v,w)-v-w$ is connected. Let $P$ be a 
path of $G$ from the vertex $x$ to the vertex $y$. Let $o$ be the cycle obtained from the path $P$ 
by 
adding the vertex $v$. By {\cite[\autoref*{cycle_gen}]{{loc2sepr}}} and 
{\cite[\autoref*{gen}]{{loc2sepr}}}, there is a set $\Ccal$ of cycles of length at most $r$ of 
$G$ 
generating the cycle $o$. 
As the geodesic cycles generate all cycles, we pick the set $\Ccal$ so that it only contains 
geodesic cycles. As the graph $G$ is triangular, all cycles in the set $\Ccal$ are triangles. 

Now let $k$ be the connected component of the graph $H$ containing the vertex $vx$. Each triangle 
in 
the set $\Ccal$ containing the vertex $v$ or $w$ gives rise to precisely one edge of the graph 
$H$. 
So 
each 
of the generating triangles contains an even number of edges that are vertices of the component 
$k$. As the cycle $o$ 
is equal to the sum over the cycles in $\Ccal$ -- evaluated over the field $\Fbb_2$ --, it also 
must contain an even number of edges that are 
vertices of the component $k$. By construction, it contains precisely two edges that are 
vertices of the auxiliary graph 
$H$; these are the edges $vx$ and $vy$ of the graph $G$. 
And the vertex $vx$ is in the component $k$ by construction. So also the vertex $vy$ of $H$ must be 
in the component $k$.
Thus the vertices $vx$ and $vy$ are in the same component of the graph $H$. 
\end{proof}

An edge of the auxiliary graph $H$ between vertices $e_1$ and $e_2$ is \emph{green} if all geodesic 
cycles of the graph $G$ of 
length at most $r$ containing $e_1$ and $e_2$ contain the edge $e$. Other edges of $H$ are 
\emph{red}. We distinguish two cases.

{\bf Case 1:} there is a path in the graph $H$ from the set $\{vx,wx\}$ to the set $\{vy,wy\}$ 
consisting only of red edges.\footnote{This case includes the trivial case where $xy$ 
is an edge of $G$. We will not use this in our argument.} Let $o_1,..., o_n$ be cycles giving rise 
to the 
 edges of that path. As all these edges are red, we may pick, and we do pick, the cycles $o_i$ such 
that they do not contain the edge $e$. 
As $G$ is triangular, all these cycles are triangles. So none of these triangles 
can use both of the vertices $v$ and $w$. So there is one of these vertices, say $v$, that is a 
vertex of all these triangles. 
Thus the vertex $v$ is the center of a wheel, whose pieces are the triangles $o_1$,..., $o_n$, 
$vxw$ and $vyw$. As all its pieces have length at most $r$, this wheel is an $r$-weighted 
wheel. 

{\bf Case 2:} not Case 1; that is, there is a cut of the graph $H$ separating the set $\{vx,wx\}$ 
from
$\{vy,wy\}$ that consists only of green edges. This cut must be nonempty by \autoref{sublem21}.

Let $g$ be a green edge in such a cut such that 
there is a red path (that is a path all whose edges are red) from  $\{vx,wx\}$ to an endvertex of 
$g$. Denote this path by $P$ and the endvertex of $P$ by $z$. For later reference, we point out 
that we can pick, and we do pick, the edge $g$ from a cut separating $\{vx,wx\}$ 
from
$\{vy,wy\}$ that is minimal; in particular $g$ is not equal to the edge between the vertices $vx$ 
and $wx$.

Let $o_1,..., o_n$ be cycles giving 
rise to the 
 edges of that path. As all these edges are red, we may pick, and we do pick, the cycles $o_i$ such 
that they do not contain the edge $e$. 
As $G$ is triangular, all these cycles are triangles. So none of these triangles 
can use both of the vertices $v$ and $w$. So there is one of these vertices, say $v$, that is a 
vertex of all these triangles. 
Let $o$ be a cycle giving rise to the edge $g$. 

We claim that the vertex $v$ is the center of a wheel, whose pieces are the triangles $o_1$,..., 
$o_n$, 
$o$ and $vxw$. As all these cycles have length at most $r$, this wheel would be an $r$-weighted 
wheel. 
So all that remains to show is that there are at least three pieces; that is, that there is at 
least one cycle $o_i$. In other words, the path $P$ must not consists of a single vertex. This is 
not the case, indeed, as the triangle $o$ would then share two edges with the triangle $vxw$. 
Then the triangles  $o$ and $vxw$ would be identical, so the cycle $o$ would give rise to the edge 
between the vertices $vx$ and $wx$. As pointed out above, this violates the choice of the edge $g$.
Hence there is at least one cycle $o_i$, and so $o_1$,..., 
$o_n$, 
$o$ and $vxw$ form the pieces of an $r$-weighted 
wheel centered at the vertex $v$.
\end{proof}

\subsection{Constructing fans}\label{subsec:fan}

A \emph{fan} of parameter $r$ centered around a vertex $v$ is a sequence of oriented cycles 
$o_1,...,o_n$ that all have length at most $r$ and contain the vertex $v$ that satisfy the 
following.
\begin{enumerate}
 \item The intersection $o_i\cap o_j=\{v\}$ if $|i-j|\geq 2$; and 
 \item each cycle $o_i$ has the form $vL_iM_iR_iv$ for subpaths $L_i$, $M_i$ and $R_i$ such that 
$R_{i}=L_{i+1}= o_{i+1}\cap o_{i}$ for all $i\in [n-1]$.
\end{enumerate}
See \autoref{fig:fan}. 
We refer to the vertex $v$ as the \emph{center} of the fan. The \emph{start} (or starting edge) of 
a fan is the first edge of the directed path $L_1$ and the \emph{end} (or ending edge) of the fan 
is the last edge of directed path $R_n$.  Note that the starting edge and ending edge are always 
incident with the center of the fan.

   \begin{figure} [htpb]   
\begin{center}
   	  \includegraphics[height=3cm]{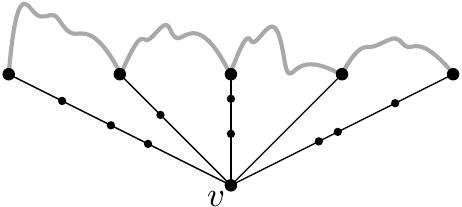}
   	  \caption{A fan centered at the vertex $v$. The paths $M_i$ are highlighted in 
grey.}\label{fig:fan}
\end{center}
   \end{figure}

A \emph{pre-fan} of parameter $r$ centered around the vertex $v$ is a sequence of oriented 
cycles 
$o_1,...,o_n$ that all have length at most $r$ and contain the vertex $v$ 
such that the vertex just after $v$ on $o_i$ is equal to the vertex just before $v$ on $o_{i+1}$ 
(for $i\in [n-1]$). The \emph{start} of a pre-fan is the edge of $o_1$ just before $v$ 
and the \emph{end} of the fan is the edge just after $v$ on $o_n$.  

\begin{eg}
 Every fan is a pre-fan. 
\end{eg}

Roughly speaking, the next lemma gives a way how pre-fans can be `improved to' fans. 

We say that a pre-fan $(v, o_1,...,o_n)$ \emph{contains} another pre-fan $(v, o_1',...,o_n')$ if 
$\bigcup o_i\supseteq \bigcup o'_i$
; 
that is, the first pre-fan contains the second as a subgraph. 
Given a pre-fan $F=(o_1,...,o_n)$, we refer to the cycles $o_i$ as the \emph{pieces} of the pre-fan 
$F$.

\begin{lem}\label{pre-fan_to_fan}
Every pre-fan $\Fcal'$ of parameter $r$ centered at $v$ contains a fan $\Fcal$ of parameter $r$ 
centered at $v$ that 
has the same start and end.

Moreover, if the pre-fan $\Fcal'$ has at least two pieces and the start 
is only an edge of its 
first cycle, and its end is only an edge of its last cycle, then the fan $\Fcal$ has at 
least two pieces.
\end{lem}

\begin{proof}
Let  $(v, o_1,...,o_n)$ be a pre-fan of parameter $r$ centered at $v$. We consider the following 
operation. Assume two cycles $o_i$ and $o_j$ with $i<j$ share a vertex $x$. Let $R_i$ be the path 
from $x$ to $v$ in the cyclic orientation of oriented cycle $o_i$, and let $L_j$ be the path from 
$v$ to $x$ in the cyclic orientation of the oriented cycle $o_j$. 
We obtain $o_i'$ from $o_i$ by replacing the path $R_i$ by $L_j$.
Similarly, we obtain $o_j'$ from $o_j$ by replacing the path $L_j$ by $R_i$.
If the path $R_i$ is not longer 
than the path $L_j$, then $(v, o_1,...,o_i,o_j',o_{j+1},...,o_n)$ is a pre-fan of parameter $r$ 
centered at $v$. Otherwise $(v, o_1,...,o_{i-1},o_i',o_j,...,o_n)$ is a pre-fan of parameter 
$r$ 
centered at $v$. We refer to this new pre-fan as the \emph{reduction} of the pre-fan $(v, 
o_1,...,o_n)$ at the vertex $x$ along the indices $i$ and $j$. Clearly a reduction of a pre-fan 
$\Fcal'$ has the same start and end as $\Fcal'$ and is contained in $\Fcal'$.
 
Fans are fixed points for the reduction operation. In fact, it will follow from this proof that 
they 
are the only fixed points. This might suggest the following proof strategy.

Given a pre-fan, we shall iteratively apply 
reductions to it. We shall show that this procedure eventually stops, and that when it stops we 
 have reduced the pre-fan to a fan.

\begin{sublem}\label{reduction_type_1}
Assume $(v,\hat o_1, ...\hat o_m)$ is obtained from $(v,o_1, ...o_n)$ by a reduction at a vertex 
$x$ along indices $i<j-1$. Then $n>m$. 
\end{sublem}
\begin{proof}
By construction $m=n-(j-i+1)$. 
\end{proof}

By \autoref{reduction_type_1}, each time we perform a reduction along indices with distance at 
least two, the number of cycles shrinks. Hence eventually such reductions must be no longer 
possible; that is, cycles $o_i$ and $o_j$ with $i<j-1$ can only intersect in the vertex $v$. 
For $i\in [n-1]$, let $R_i$ be a maximal subpath of the oriented cycle $o_i$ ending at the vertex 
$v$ such that (the reverse of the directed path) $R_i$ is a subpath of the oriented cycle $o_{i+1}$.
Let $L_{i+1}=R_i$. Let $M_i$ be the subpath of the oriented cycle $o_i$ from the last vertex of the 
path $L_i$ to the first vertex of the path $R_i$. 

We have shown that the cycle $o_i$ cannot intersect the cycle $o_{i+1}$ in the subpath $L_i-v$ (as 
this would intersect the cycle $o_{i-1}$). If they intersect in an interior point of the path 
$M_i$, applying a reduction at this point increases the length of the path $R_i$ but leaves all  
paths $R_k$ with $k\neq i$ invariant. Also cycle length cannot increase during such a reduction. 
Hence we can only perform a bounded number of such reductions. After such reductions are no 
longer possible, the intersection of the cycles $o_i$ and $o_{i+1}$ is precisely $R_i$. 
Doing this analysis for all indices $i\in [n-1]$ yields that we have a fan.

To see the `Moreover'-part note that under these assumptions our constructions ensure that the 
first piece never contains the ending edge and the last piece never contains the starting edge. 
Thus 
the final fan needs to have at least two pieces. 
\end{proof}

Roughly speaking, the next lemma gives conditions under which pre-fans exist.

\begin{lem}\label{pre-fan-exists}
Let $G$ be a (weighted) $r$-locally $3$-connected graph. 
Assume there is a cycle $o'$ containing vertices $v_0$ and $v_1$ that are not adjacent on the 
cycle $o'$. Assume $o'$ 
includes a shortest path $P$ from $v_0$ to $v_1$.

 There is a pre-fan of 
parameter $r$ centered at some vertex $v_i$ none of whose pieces contains the vertex $v_{i+1}$ (for 
some $i\in \Fbb_2$). And there is a cycle $o$ of length at most $r$ including $P$ such that the 
start and end of the pre-fan are on the cycle $o$. 

Moreover, if $o\neq o'$, then  $o$ is geodesic.
\end{lem}

\begin{rem}
 If all edges of the graph $G$ have length one, then the vertices $v_0$ and $v_1$ cannot be 
adjacent in the graph $G$. Indeed, then the path $P$ would have length one and thus $v_0$ and 
$v_1$ would be adjacent on the cycle $o'$.  
\end{rem}

\begin{proof}[Proof of \autoref{pre-fan-exists}.]
A cycle $R$ of the explorer-neighbourhood 
$\expl(v_0,v_1)$ is \emph{valid} if it contains the vertex $v_0$ but not the vertex $v_1$, and  it 
contains the two neighbours of $v_0$ on $o'$.

\begin{sublem}\label{valid_exists}
 There is a valid cycle. 
\end{sublem}

\begin{proof}
As the graph $G$ is $r$-locally $3$-connected, the punctured explorer-neighbourhood  
$\expl(v_0,v_1)-v_0-v_1$ is connected. So there is a path $Q$ of $\expl(v_0,v_1)-v_0-v_1$ joining 
the two neighbours of the vertex $v_0$ on the cycle $o'$, which by assumption are both different 
from the vertex $v_1$. 
Then $Q+v_0$ is a valid cycle. 
\end{proof}

A cycle is \emph{friendly} if it has
length at most $r$, and if it contains both vertices $v_i$, then it has the path $P$ as 
a subpaths; moreover we require that it is geodesic.

\begin{sublem}\label{is_friendly}
 Any valid cycle has a generating set of friendly cycles. 
\end{sublem}

\begin{proof}
Let $Q'$ be a valid cycle. By 
{\cite[\autoref*{cycle_gen}]{{loc2sepr}}} and {\cite[\autoref*{gen}]{{loc2sepr}}} $Q'$ is 
generated by cycles of length at most 
$r$. 
We shall show by induction that any cycle $x$ of length at most $r$ has a generating set of 
friendly cycles. 
If $x$ is not geodesic, it is generated by two shorter cycles, which by induction have a generating 
set of friendly cycles. 
Hence we may assume, and we do assume, that the cycle $x$ contains the vertices  $v_0$ and 
$v_1$ but not the path $P$.
Denote by $X_1$ and $X_2$ the two subpaths of the cycle $x$ from $v_0$ to $v_1$. As $P$ is a 
shortest path 
from $v_0$ to $v_1$, the closed walks $PX_1$ and $PX_2$ both have length at most the length of 
$x$. 
If one closed walk $PX_i$ is not a cycle, it is an edge-disjoint union of cycles,
and all cycles in this union have strictly fewer edges than the original cycle $x$.
If $PX_i$ is a non-geodesic cycle, it is generated by two shorter cycles.
To summarise, either $PX_i$ is a friendly cycle or by induction it is generated by a set of 
friendly cycles.
Thus we have shown that any cycle $x$ in the generating set for the valid cycle $Q'$ has a 
generating set of friendly cycles, so $Q'$ has a generating set of friendly cycles.
\end{proof}

By \autoref{valid_exists} and \autoref{is_friendly}, there is a valid cycle with a 
generating set of friendly cycles.
Call such a valid cycle $R$. Denote a generating set of $R$ consisting of friendly cycles by 
$\Ccal$. 

By $\Dcal$ we denote the set of all $c\in \Ccal$ 
that contain both vertices $v_0$ and $v_1$. 
We distinguish two cases.

{\bf Case 0:} $|\Dcal|$ is even. 

Now define the following auxiliary graph. Its vertex set are the cycles of $\Ccal\sm \Dcal$ that 
contain the vertex $v_0$, 
the two subpaths of $o'$ between $v_0$ 
and $v_1$, and for each $d\in \Dcal$ we take the $v_0$-$v_1$-path $d\sm P$. 
We denote the set of these paths by $\Dcal'$. 
We add an 
edge 
between any two vertices of the auxiliary graph whose corresponding cycles or paths share an edge 
incident with the vertex $v_0$. 
We refer to this auxiliary graph as $A_0$. 

\begin{sublem}\label{path_exists0}
 In the graph $A_0$, there is a path from the vertex $P$ to some other vertex in $\Dcal'\cup\{o'\sm 
P\}$. 
\end{sublem}

\begin{proof}
Let $K$ be the component of the graph $A_0$ containing the vertex 
$P$. Denote by $e$ the edge of $P$ incident with its endvertex $v_0$.
Consider the sum $S$ over all $c\in \Ccal\sm \Dcal$ that are in the component $K$ (over 
$\Fbb_2$). This sum has even degree at every vertex -- in particular the vertex $v_0$. 
As $|\Dcal|$ is even, and the sum over all $c\in \Ccal$ is nontrivial at the edge $e$, also the 
sum $S$ is nontrivial at the edge $e$. So there must be another edge incident with the vertex $v_0$ 
at 
which the sum $S$ is non-zero. Denote such an edge by $f$. 

If the edge $f$ is in the cycle $R$, then $f$ is on the cycle $o'$. As the edge $f$ is not on the 
path $P$, it must be on the path $o'\sm P$. Thus the vertex $o'\sm P$ of the graph $A_0$ would be 
in the 
component $K$, as desired. Otherwise the edge $f$ is not on $R$. So it must be on some $d\in \Dcal$.
As the edge $f$ is not on the 
path $P$, it must be on the path $d\sm P$. Thus the vertex $d\sm P$ of $A_0$ would be in the 
component $K$, which completes the proof. 
\end{proof}

By \autoref{path_exists0}, there is a path $Q$ in the graph $A_0$ from $P$ to some vertex in 
$\Dcal'\cup\{o'\sm 
P\}$. 
We pick a shortest such path. Then all interior vertices have their associated cycles in the set 
$\Ccal\sm \Dcal$. 
These cycles form the 
pieces of a pre-fan centered at the vertex $v_0$ of parameter $r$. Denote that pre-fan by $F$. 
If the endvertex of the path $Q$ is $o'\sm P$, we pick $o=o'$. Otherwise there is some $d\in \Dcal$ 
such that $d\sm P$ is the endvertex of the path $Q$. We pick $o=d$. By construction, the path $P$ 
is included in the cycle $o$, and this cycle has length at most $r$. Also the starting edge and 
the ending edge of the pre-fan $F$ are on $o$.  
As none of the cycles of the pre-fan contains the vertex $v_1$ by construction, this completes 
this case. 

Moreover, if $o\neq o'$, then $o\in \Ccal$ and so it is geodesic. 

{\bf Case 1:} $|\Dcal|$ is odd. 
This case is somewhat similar to Case 0 with the roles of `$v_0$' and `$v_1$' interchanged.

The graph $A_1$ is defined like the graph $A_0$ with the vertex `$v_1$' in place of the vertex 
`$v_0$'. We define $\Dcal'$ as in Case 0. 
Similarly as \autoref{path_exists0} one proves the following.

\begin{sublem}
 In the graph $A_1$, there is a path from the vertex $P$ to some other vertex in $\Dcal'\cup\{o'\sm 
P\}$. 
\end{sublem}

\begin{proof}
This is the same as the proof of \autoref{path_exists0} with `$v_1$' in place of `$v_0$' and 
`$A_1$' 
in 
place of `$A_0$', and `odd' in place of `even'. In particular, we argue as follows: `As $|\Dcal|$ 
is odd, and the sum over all $c\in \Ccal$ is trivial at the edge $e$, the 
sum $S$ is nontrivial at the edge $e$.' Indeed, unlike in Case 0, here $e\notin R$. 
\end{proof}

The rest of this case is the same as in Case 0 with `$v_1$' in place of `$v_0$' and `$A_1$' in 
place of `$A_0$'.

\end{proof}

\subsection{Detecting bounded subdivisions of the 3-wheel or the 4-wheel}\label{det:bd_wheel}

The purpose of this subsection is to prove \autoref{theta_short-path} stated below.

A \emph{theta-graph} (of parameter $r$) is a graph $\theta$ consisting of two vertices $v$ and 
$w$ and three internally disjoint paths between them such that at least two of the cycles obtained 
by concatenating a pair of these paths have length at most $r$. The vertices $v$ and $w$ are 
referred to as the \emph{branching vertices} of $\theta$. The three paths from $v$ to $w$ are 
referred 
to as \emph{arms}, see \autoref{fig:theta}.

   \begin{figure} [htpb]   
\begin{center}
   	  \includegraphics[height=3cm]{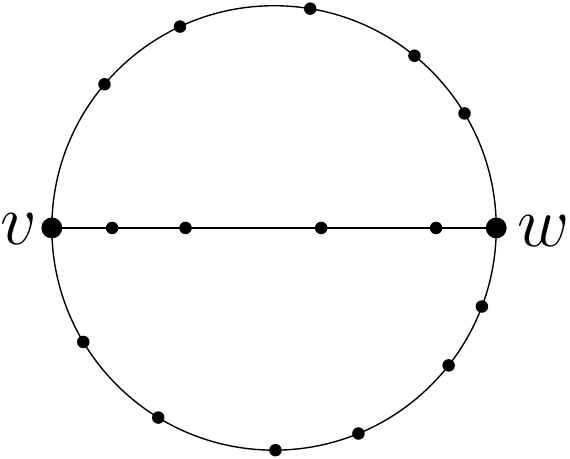}
   	  \caption{A theta-graph with branching vertices $v$ and $w$. }\label{fig:theta}
\end{center}
   \end{figure}

\begin{lem}\label{theta_short-path}
Let $G$ be a graph that contains a theta-graph $\theta$ of parameter 
$r$ such 
that there is a cycle $o$ of length at most $r$ containing a branching vertex of $\theta$ and a 
path $P$ between interior vertices of different arms of $\theta$ that avoids the branching vertices 
of 
$\theta$.
Then $G$ contains an $r$-local subdivision of the $3$-wheel or the $4$-wheel.
\end{lem}

\begin{eg}
 The assumption that the cycle $o$ contains a branching vertex cannot be omitted. An example 
demonstrating this is depicted in \autoref{fig:not3con}. 
Indeed, this graph has a theta-subgraph of parameter $r$, whose two branching vertices are denoted 
by $x$ and $y$ in the figure. This subgraph is obtained from the graph depicted by deleting the 
two edges of length $r'=(r/2)-1$ that are incident with neither $x$ nor $y$. This graph does not 
contain an $r$-local subdivision of a wheel as the graph obtained by $r$-locally cutting at the 
$r$-local $2$-separator $\{x,y\}$ is a series-parallel graph. Hence this graph has a 
graph-decomposition of locality $r$ and width two. 
\end{eg}

      \begin{figure} [htpb]   
\begin{center}
   	  \includegraphics[height=3cm]{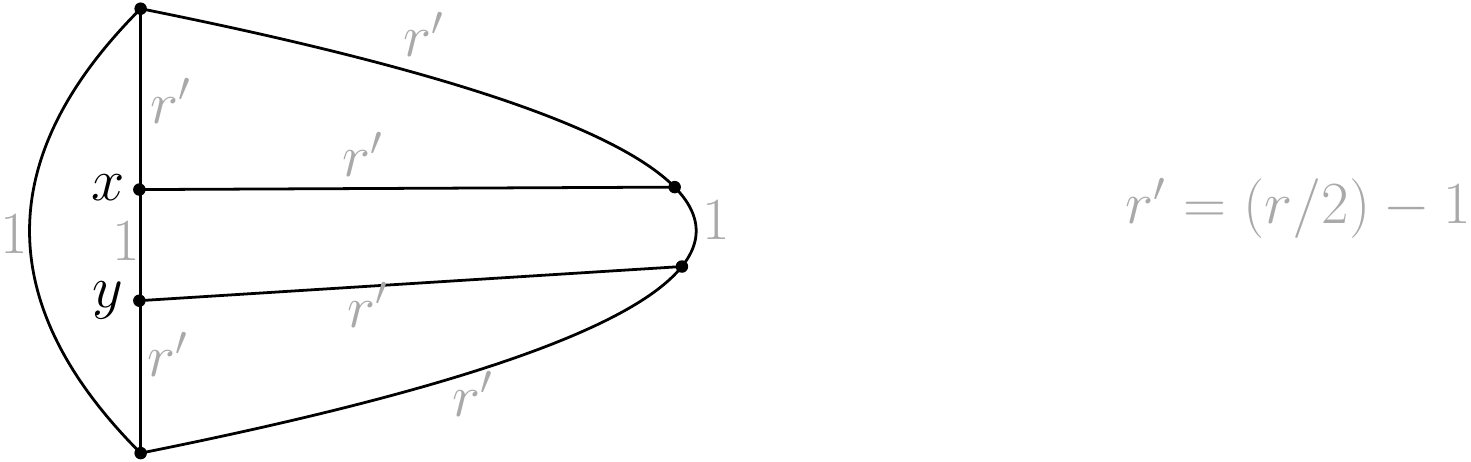}
   	  \caption{A graph that is 3-connected but not $r$-locally $3$-connected as 
can be seen by considering the explorer-neighbourhood $\expl(x,y)$.   
   	  Lengths of edges are given in grey.}\label{fig:not3con}
\end{center}
   \end{figure}

Instead of proving \autoref{theta_short-path} directly, we shall deduce it from a variant, 
\autoref{theta_short-path_tech} stated below. Next we develop the context of 
\autoref{theta_short-path_tech}.
   
A graph $G$ is in the class $\Wcal$ if it is obtained from a theta-graph $\theta$ of parameter 
$r$ by attaching\footnote{Here \emph{attaching} a path means that we add this path disjointly 
and then identify its endvertices as prescribed.} a path $P$ at interior 
vertices of different arms of the theta-graph $\theta$ such that $G$ contains a cycle of length at 
most $r$ including $P$. 
A \emph{weighted suppression} of a graph is obtained by
iteratively suppression vertices of degree two. Here the length of the suppression edge is the sum 
of the length of the two edges incident with the suppressed vertex.

\begin{lem}\label{auxi1}
Any graph in $\Wcal$ is a subdivision of the graph $K_4$ such that there is a spanning tree all 
whose fundamental cycles have length at most $r$.

In particular, weighted suppressions of graphs in $\Wcal$ are $r$-locally $3$-connected.
\end{lem}

\begin{proof}
Clearly any graph in the class $\Wcal$ is a subdivision of the graph $K_4$. It remains to 
construct a spanning tree such that all its fundamental cycles have length at most $r$. 
Start with a spanning tree of the theta-graph such that all its fundamental cycles have length at 
most $r$. Now extend it arbitrarily to a spanning tree of the whole graph by adding edges. 

The `In particular'-part follows from the fact that the fundamental cycles of any spanning tree 
generate all cycles and \autoref{gen_loc_con}. 
\end{proof}

A graph $G$ is in the class $\Wcal^*$ if it is obtained from a theta-graph $\theta$ of parameter 
$r$ by attaching a path $P$ at interior 
vertices of different arms of the theta-graph $\theta$ and another path $Q$ at a branching vertex 
and some interior vertex of an arm that also contains a vertex of $P$ such that in this arm this 
endvertex of $P$ is in between the two endvertices of $Q$
-- in such a way that $G$ contains a cycle of length at 
most $r$ including both $P$ and $Q$, see \autoref{fig:construct_spanning_tree}. 

\begin{lem}\label{auxi2}
Any graph in $\Wcal^*$ is a subdivision of the $4$-wheel and all its cycles are generated by 
cycles of length at most $r$.

In particular, the weighted suppressions of graphs in $\Wcal^*$ are $r$-locally 
$3$-connected. 
\end{lem}

\begin{proof}
Let $G$ be a graph in the class $\Wcal^*$. Remove the branching vertex of the theta-graph that 
has degree four in $G$. Call that vertex $v$. Then remove all vertices of degree one iteratively. 
The resulting graph is a cycle $o$. In the graph $G$, there are four path from $o$ to $v$ that only 
intersect at the common vertex $v$. Thus $G$ is a subdivision of a $4$-wheel.    

Let $o_1$ and $o_2$ be two cycles of the theta-graph of length at most $r$. Note that $o_1$ and 
$o_2$ together cover all edges of the theta-graph. 
By assumption there is a cycle $u$ of length at most $r$ including the paths $P$ and $Q$. 
Let $x$ be the endvertex of the path $Q$ different from the branching vertex $v$. 
One of the cycles $o_1$ or $o_2$, say $o_1$, contains the vertex $x$. 
So the cycle $o_1$ contains a path of length at most $r/2$ between the vertices $v$ and $x$. Denote 
this path by 
$R_1$. The cycle $u$ contains a path of length at most $r/2$ between $v$ and $x$. Denote that path 
by $R_2$. 
As the path $R_2$ contains precisely one of the paths $P$ and $Q$, it is distinct from the path 
$R_1$. Thus the closed walk obtained by concatenating the paths $R_1$ and $R_2$ includes a cycle. 
denote that cycle by $o_3$.    See \autoref{fig:construct_spanning_tree}.

   \begin{figure} [htpb]   
\begin{center}
   	  \includegraphics[height=4cm]{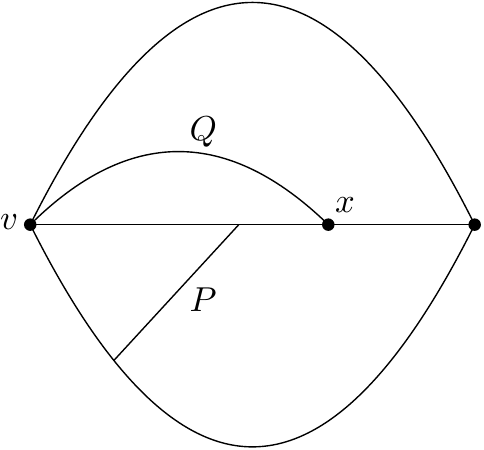}
   	  \caption{A graph in the class $\Wcal^*$. 
}\label{fig:construct_spanning_tree}
\end{center}
   \end{figure}

\begin{sublem}\label{generate}
 The cycles $o_1$, $o_2$, $o_3$ and $u$ generate all cycles of the graph $G$.
\end{sublem}

\begin{proof}
The cycles $o_1$ and $o_2$ generate all cycles of the theta-graph. 
There is a unique arm of the theta-graph containing the two endvertices of the path $Q$.
Let $o_3'$ be the cycle obtained 
from the path $Q$ by joining its two endvertices in that arm. The cycle $o_3'$ is generated by the 
cycles $o_3$, $o_1$ and $u$. Let $u'=u+o_3'$, which is a cycle containing the path $P$ but not the 
path $Q$. 

The cycles $o_1$, $o_2$, $o_3'$ and $u'$ clearly generate all cycles of $G$. Hence $o_1$, $o_2$, 
$o_3$ and $u$ generate all cycles of the graph $G$. 
\end{proof}

The `In particular'-part follows from \autoref{gen_loc_con}. 
\end{proof}

\begin{lem}\label{theta_short-path_tech}
 Let $G$ be a graph as in \autoref{theta_short-path}.

Then $G$ contains a graph in the class $\Wcal\cup \Wcal^*$.
\end{lem}

\begin{proof}[Proof that \autoref{theta_short-path_tech} implies \autoref{theta_short-path}.]
 This is a direct consequence of \autoref{auxi1} and \autoref{auxi2}.  
\end{proof}

We prepare to prove \autoref{theta_short-path_tech}.  Let $\theta$, $o$ and $P$ be as in 
\autoref{theta_short-path}.

An \emph{arc} is a nontrivial\footnote{A path is \emph{nontrivial} if it contains at least one 
edge.} subpath $Q$ of $o$ such that $Q$ intersects the theta-graph $\theta$ 
precisely in its 
endvertices. We say that $Q$ is a 
($\theta$-)\emph{bridge} if its endvertices are interior vertices of different arms of the 
theta-graph $\theta$. Otherwise, we say that $Q$ is a ($\theta$-)\emph{detour}. Note that the path 
$P$ from the assumption includes a bridge, so there is at least one bridge.

\begin{rem}
 In this proof we will step by step improve the theta-graph $\theta$ and the cycle $o$. Roughly 
speaking, this means that we will eliminate the arcs one by one -- until at most two of them are 
left over. Then we will find a configuration in the classes $\Wcal$ or $\Wcal^*$.
We start by describing relevant properties of arcs. 
\end{rem}

No two arcs have adjacent internal vertices (here a vertex of a path is internal if it is not an 
endvertex). 
We say that an arc $R$ is \emph{adjacent} to an arc $S$ if there is a subpath $X$ of $o$ joining 
some of their endvertices such that this subpath does not contain any internal vertices of arcs. 
\begin{eg}\label{three}
 If there are at least three arcs, each arc is adjacent to precisely two other arcs. These two arcs 
are distinct. 
\end{eg}
We say that an arc $R$ is \emph{weakly adjacent} to an arc $S$ if it is adjacent to $S$ and the 
endvertex of $R$ on an 
arc $X$ witnessing adjacency is a branching vertex of $\theta$. It is \emph{strongly adjacent} if  
 the endvertex of $R$ on such a path $X$ is not a branching vertex of $\theta$.
\begin{eg}
An adjacent arc that is not weakly adjacent is strongly adjacent. 
 If there are at least three arcs, no two arcs can be strongly adjacent and weakly adjacent. 
\end{eg}

The endvertices of a detour $Q$ are contained in a single arm. This arm is uniquely determined 
-- 
unless the two endvertices of the detour are the two branching vertices of $\theta$. In this case, 
we choose the unique arm that does not contain any endvertex of the path $P$. We 
refer to this uniquely defined arm as the arm \emph{circumvented} by $Q$ (relative to the bridge 
$P$). 
The \emph{replacement path} of $Q$ is the unique subpath of the circumvented arm between the two 
endvertices of $Q$. We denote the replacement path of $Q$ by $Q'$. 

Given a bridge $B$, a detour $Q$ is \emph{$B$-free} if its replacement path $Q'$ has no internal 
vertex that is an 
endvertex of $B$. 

\begin{lem}\label{detour_nice}
Assume there is exactly one bridge $B$, and at least one detour.
Then there is a detour that is $B$-free  or that is strongly adjacent to $B$ --
or else $\theta\cup o$ is in the class $\Wcal^*$. 
\end{lem}

\begin{proof}
By assumption there is a detour $Q$ that is adjacent to the bridge $B$. We may assume, and we do 
assume, that the detour $Q$ is weakly adjacent and not $B$-free. In particular, one endvertex 
of $Q$ is not a branching vertex of $\theta$. 
Thus if $B$ and $Q$ are the only arcs, then the graph  $\theta\cup o$ is in the class 
$\Wcal^*$. 
Hence we may assume, and we do assume, that there is another detour. By \autoref{three}, there is a 
detour $R$ adjacent to $B$ that is distinct from $Q$. By relabelling the branching vertices if 
necessary, we may assume, and we do assume, that the branching vertex $v$ is an endvertex of the 
detour $Q$.

   \begin{figure} [htpb]   
\begin{center}
   	  \includegraphics[height=3cm]{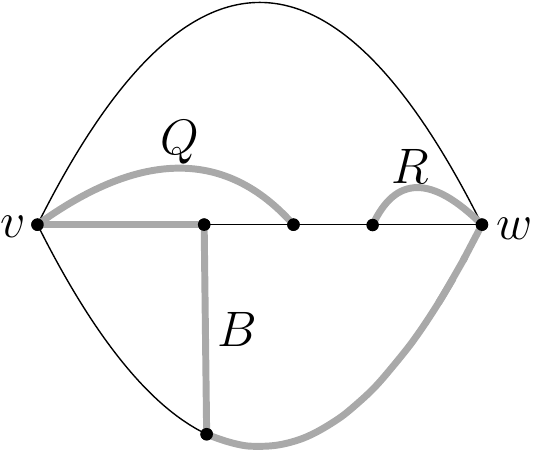}
   	  \caption{The path $S$ is highlighted in grey.}\label{fig:long_path}
\end{center}
   \end{figure}

Next we construct a subtrail $S$ of the cycle $o$, see \autoref{fig:long_path} (this subtrail of 
$o$ is a subpath of $o$ or equal to $o$). Start the 
detour $Q$ at the endvertex different 
from the branching vertex $v$ all the way to the vertex $v$, then follow the replacement 
path 
$Q'$ until you hit an endvertex of $B$ (which must happen eventually as the detour $Q$ is not 
$B$-free by assumption), take the bridge $B$, follow the cycle $o$ until you hit the detour $R$, 
then follow $R$. Denote this trail by $S$. 

We refer to the first vertex of the trail $S$ on the detour $R$ by $w$. If the vertex $w$ is not a 
 branching vertex, 
then $R$ is strongly adjacent to $B$, and we are done.
Hence we may assume, and we do assume, that $w$ is a branching vertex. 
As $S$ is a trail, its interior vertices $v$ and $w$ cannot be identical. 

Hence the two endvertices of the trail $S$ cannot be branching vertices, and thus are 
interior vertices of arms of $\theta$. As the path $o\sm S$ (between the two endvertices of 
$S$) 
does not include a bridge and no branching vertex of $\theta$, this path must have both its 
endvertices on a single arm of $\theta$. 

Denote the endvertex of the detour $R$ different from 
$w$ by $z$. 
By the above, the vertex $z$ is on the arm of the theta-graph $\theta$ circumvented by $Q$.
Note that the replacement path $Q'$ is equal to the subpath of this arm from 
$v$ to the endvertex of $B$ on that arm. 
As $S$ is a trail, its endvertex $z$ cannot be on the subpath $Q'$. 
So the replacement path for $R$ is disjoint from the path $Q'$, and so in particular does not 
contain the endvertex of the bridge $B$ on the arm circumvented by $R$. Thus the detour $R$ 
is $B$-free. This completes the proof. 
\end{proof}

\begin{lem}\label{single_school}
 If there is at most one bridge, then $\theta\cup o$ has a subgraph 
in the class $\Wcal\cup \Wcal^*$. 
\end{lem}

\begin{proof}
From the graph $H=\theta\cup o$ we pick a theta-graph $\theta'$ of parameter $r$ and a 
cycle $o'$ of length at most $r$ containing a branching vertex of $\theta'$ so that $o'$ has only a 
single bridge. 
Such 
a choice is possible as we could simply take $\theta$ and $o$. 
Now we pick $o'$ 
and $\theta'$ amongst all possible choices so that there are as few detours as possible. 
By replacing `$(\theta,o)$' by `$(\theta',o')$' if necessary, we may assume, and we do assume, that 
$(\theta,o)$ has as few detours as possible. Denote the unique $\theta$-bridge included in $o$ 
by $P$. Assume that the graph $H$ has no subgraph in the class $\Wcal^*$. 

Suppose for a contradiction, there is a $\theta$-detour included in the cycle $o$. Let $Q$ be a 
detour. By \autoref{detour_nice}, we may assume, and we do assume, that $Q$ is $P$-free 
or strongly adjacent to $P$. 
We denote the arm circumvented by $Q$ by $A$, and the replacement path for $Q$ by $Q'$. 

We distinguish two cases. 

{\bf Case 1:} the length of $Q'$ is at most the length of $Q$. We obtain $o'$ from $o$ by 
replacing the path $Q$ by $Q'$. As the bridge $P$ and the detour $Q$ are internally disjoint, 
the closed walk $o'$ includes the path $P$. 
We obtain $o''$ from $o'$ by taking a cycle included in the closed walk 
$o'$ that includes the path $P$. This ensures that $P$ is a $\theta$-bridge included in $o''$. 
Now let $R$ be a subpath of $o''$ that intersects $\theta$ precisely at its endvertices. As $o''\sm 
\theta\se o$, this subpath $R$ is a subset of the cycle o. Thus 
$R$ is equal to the bridge $P$ or a detour included in $o$. 
We conclude that $P$ is the only $\theta$-bridge included in $o''$ and there are strictly less 
$\theta$-detours included in $o''$. It follows that the path $o''\sm P$ includes a branching 
vertex of $\theta$. 
Hence the pair $(\theta,o'')$ contradicts the minimality of $(\theta,o)$. Thus we get a 
contradiction in this case. 

{\bf Case 2:} the length of $Q'$ is strictly larger than the length of $Q$.
We obtain $\theta'$ from the theta-graph $\theta$ by replacing the path $Q'$ by $Q$.
As the path $Q$ does not contain any other vertices of $\theta$ except for its endvertices, the 
graph $\theta'$ is a theta-graph. And its parameter is at most that of $\theta$. Note that the 
theta-graphs 
$\theta$ and $\theta'$ have the same branching vertices. 
In particular, the cycle $o$ includes a branching vertex of the theta-graph $\theta'$. 
By our choice of $Q$ according to 
\autoref{detour_nice}, we have to consider the following two subcases. 

{\bf Case 2A:} the detour $Q$ is $P$-free. 
As the bridge $P$ and the 
detour $Q$ are internally disjoint, the path $P$ is a $\theta'$-bridge included in $o$. And it is 
the only $\theta'$-bridge.
All $\theta'$-detours include some $\theta$-detour aside from $Q$. Thus there are less 
$\theta'$-detours than $\theta$-detours. This gives a 
contradiction to the choice of $(\theta,o)$ in this subcase. 

{\bf Case 2B:} the detour $Q$ is strongly adjacent to $P$. 
We obtain $P'$ from $P$ by adding a subpath $X$ from an endvertex of $P$ to an endvertex of $Q$ 
witnessing that $P$ and $Q$ are strongly adjacent. Hence the path $P'$ does not contain any 
branching vertex, and so is a $\theta'$-bridge. All other $\theta'$-arcs are included in $o$ and 
are disjoint from the subpath $X$ of $o$, and include 
$\theta$-arcs. These $\theta$-arcs are $\theta$-detours. 
So $P'$ is the only $\theta'$-bridge. 
As $Q$ is not included in a 
$\theta'$-arc, this gives a 
contradiction to the choice of $(\theta,o)$ in this subcase. 

\vspace{.15cm}

Having considered all cases, we conclude that there is no detour. Hence the graph $\theta\cup o$ is 
 in the class $\Wcal$. This completes the proof. 
\end{proof}

We refer to one of the branching vertices included in the cycle $o$ as $v$. 
We say that a bridge $Q$ is \emph{primary} if it has an endvertex $x$ such that a shortest path 
within the cycle $o$ from $x$ to the vertex $v$ includes the bridge $Q$. A bridge that is not 
primary 
is \emph{secondary}. 

\begin{lem}\label{desired_scool}
If there is a primary bridge $Q$, there is a cycle $u$ of length at most $r$  containing $v$ such 
that $Q$ is 
the only $\theta$-bridge of $u$. 
\end{lem}

\begin{proof}
As the bridge $Q$ is primary, it has an endvertex $x$ such that the shortest path $S$ from $x$ to 
$v$ within the cycle $o$ includes the primary bridge $Q$.
Take a shortest subpath $S'$ of the path $S$ starting at $v$ that includes a bridge. 
By minimality, the path $S'$ includes only 
a single bridge, and edges of $S'$ outside this bridge are on the theta-graph $\theta$ or in 
detours.
Denote that single bridge on $S'$ by 
$Q'$ and the endvertex of $S'$ different from $v$ by $x'$.
As the theta-graph $\theta$ has parameter $r$, it includes a path $R$ from $x'$ to $v$ of length at 
most $r/2$. The concatenation of the paths $S'$ and $R$ is a closed walk of length at most $r$ that 
traverses the bridge $Q'$ once. 
It contains the vertex $v$. 
Hence this closed walk includes a cycle $u$ including the bridge 
$Q'$. By construction, this cycle $u$ has exactly one bridge. 
\end{proof}

\begin{lem}\label{primary_school}
 There cannot be two secondary bridges. 
\end{lem}

\begin{proof}
Either there is a single vertex of the cycle $o$ that has maximum distance from the vertex $v$ in 
the cycle $o$, or there is an edge of $o$ such that its two endvertices are the only vertices with 
maximum distance from $v$. 
We 
refer to 
these vertices as \emph{barriers}. 
 As different bridges do not have adjacent interior vertices (and if there are two barriers they 
are adjacent), there can be at most one bridge that has a 
barrier as an interior vertex.

It suffices to show that any bridge $Q$ that has no barrier as an interior vertex is primary. 
Let $x$ and $y$ be the two endvertices of the bridge $Q$. 
Now consider the path $o-v$. In the middle we have the barriers. 
As the bridge $Q$ has no barrier as an internal 
vertex, the vertices $x$ and $y$ must be on the same side of the barriers on $o-v$. Take the vertex 
of $x$ or $y$ that is nearest to the barriers. The shortest path to $v$ in $o$ from that vertex 
includes the bridge $Q$. Thus $Q$ is a primary bridge. Hence all but at most one bridge is primary. 
\end{proof}
We conclude this subsection as follows.

\begin{proof}[Proof of \autoref{theta_short-path_tech}.]
Let $G$ be a graph that contains a theta-graph $\theta$ of parameter 
$r$ such 
that there is a cycle $o$ of length at most $r$ containing a branching vertex of $\theta$ and a 
path $P$ between interior vertices of different arms of $\theta$ that avoids the branching vertices 
of 
$\theta$.
Let $H$ be the subgraph of $G$ obtained by taking the union of the theta-graph $\theta$ and the 
cycle $o$. 

If there is a primary bridge, then by  
\autoref{desired_scool} 
we can modify the cycle $o$ so that there is only one $\theta$-bridge.
Otherwise by \autoref{primary_school}, there is only a single bridge. In either case, 
we may assume, and we do assume, that there is only a single bridge. By 
\autoref{single_school}, $\theta\cup o$ has a subgraph in the class $\Wcal\cup \Wcal^*$.
\end{proof}
   
As we have shown after the statement of  \autoref{theta_short-path_tech} above that it implies 
\autoref{theta_short-path}, we have also completed the proof of that lemma.

\subsection{Finding bounded wheels}\label{subsec:constr_theta}

In this subsection we prove two lemmas, which allow us to find bounded wheels in certain 
situations. They are used in the proof of \autoref{3con_to_wheel}. 

Roughly speaking, the next lemma says that we can improve a given theta-graph or else we find a 
bounded subdivision of a wheel.

\begin{lem}\label{theta_improved}
Assume $G$ has a theta-graph $\theta$ of parameter $r$ with branching vertices $v$ and $w$. 
Let $o$ be a cycle of length at most $r$ including a shortest $v$-$w$-path.
Then there is a theta-graph $\theta'$ of parameter $r$ whose branching vertices are $v$ and $w$ 
that includes the cycle $o$ -- or $G$ includes an $r$-local subdivision of a wheel.
\end{lem}

\begin{proof}
A \emph{weak $\theta$-bridge} is a path $P$ that does not contain any branching vertices of 
$\theta$ but 
joins interior vertices of different arms of $\theta$.  If the cycle $o$ included a weak 
$\theta$-bridge, 
$G$   includes an $r$-local subdivision of a wheel by \autoref{theta_short-path}.

Hence we may assume, and we do assume, that $o$ includes no weak $\theta$-bridge. Hence each of the 
two 
$v$-$w$-paths included in $o$ intersects at most one arm of $\theta$ at interior vertices. So there 
is an arm $Q$ of $\theta$ that intersects the cycle $o$ precisely in the vertices $v$ and $w$. So 
$o\cup Q$ is a theta-graph. 
By assumption, the cycle $o$ includes a shortest $v$-$w$-path; call it $R$. 
To see that $o\cup Q$ is a theta-graph of parameter $r$, we show that 
the cycle $RQ$ has length at most $r$. Let $o'$ be a cycle including $Q$ of length at most 
$r$ included in the old theta-graph $\theta$. As $R$ is a shortest $v$-$w$-path, 
the length of $RQ$ is at most that of $o'$. So $RQ$ has length at most $r$.
Thus $o\cup Q$ is a theta-graph of parameter $r$, as desired.  
\end{proof}

Roughly speaking, the next lemma gives conditions under which a bounded wheel can be built from a 
bounded fan.

\begin{lem}\label{cycle_plus_fan_to_wheel}
 Assume $G$ has a cycle $o$ of length at most $r$. 
 Assume there is a fan $F$ of parameter $r$ centered at a vertex $v$ of $o$ whose ending and 
starting edges are the two edges incident with $v$ on $o$. Assume there is a vertex $w$ of $o$ not 
contained in the fan $F$.
Assume no piece of $F$ contains interior vertices of both $v$-$w$-paths included in $o$. 

Then $G$ has an $r$-local subdivision of a wheel. 
\end{lem}

\begin{proof}
We denote the two $v$-$w$-paths included in $o$ by $P_1$ and $P_2$. 
If a piece $o_i$ of the fan $F$ contains an interior vertex of a path $P_k$, we colour that piece 
with that path $P_k$. By assumption each of piece $o_i$ is coloured with at most one path $P_k$, 
while some may not be coloured at all. 

As the first and last piece of the fan $F$ are coloured with different paths $P_k$, the fan $F$ has 
at least two pieces. Moreover, there are two pieces $o_i$ and $o_j$ with $i<j$ that are coloured 
with different paths $P_k$ such that all 
pieces $o_m$ in between (that is, with $i<m<j$) are not coloured at all.
Now we consider the subfan $o_i,...,o_j$ of the original fan $F$. Denote that subfan by $F'$.

We claim that the subgraph of $G$ that is the union of the cycle $o$ and the fan $F'$ includes 
an $r$-local subdivision of a wheel. The center of this wheel is the vertex $v$, its pieces are the 
pieces $o_m$ of the fan $F'$ with $i<m<j$, cycles $o_i'$ and $o_j'$ constructed from $o_i$ and 
$o_j$, respectively, and a cycle $o'$ constructed from $o$ as follows.

By symmetry, we assume that the piece $o_i$ contains interior vertices of the path $P_1$, 
and the piece $o_j$ contains interior vertices of the path $P_2$.
Let $a_1$ be a vertex of the piece $o_i$ on the path $P_1$ that is nearest to the vertex $w$ on 
$P_1$; note that this uniquely defines the vertex $a_1$.

As the path $P_1-v$ is disjoint from the piece $o_{i+1}$, the vertex $a_1$ is an interior vertex 
of the subpath $L_iM_i$ of the piece $o_i$. 
Let $Q_1$ be the subpath of the path $L_iM_i$ from $v$ to $a_1$.

Similarly, let $a_2$ be a vertex of the piece $o_j$ on the path $P_2$ that is nearest to the vertex 
$w$ on 
$P_2$. And let $Q_2$ be the subpath of the path $M_jR_j$ from $a_2$ to $v$.

Given $k\in \{1,2\}$, one of the paths $vP_ka_k$ and $Q_k$ is not longer than the other; pick such a 
path and denote it 
by $S_k$.

\begin{sublem}\label{rk_nice}
Given $k\in \{1,2\}$, the path $S_k$ 
intersects a piece $o_m$ with 
$i\leq m\leq j$ only in the vertex $v$ -- unless $k=1$ and $m=i$ or else $k=2$ and $m=j$. 
 The paths $S_1$ and $S_2$ intersect only at the vertex $v$. 
\end{sublem}

\begin{proof}
By symmetry, it suffices to consider the case where $k=1$. 
If the path $S_1$ is chosen to be a subpath of the path $P_1$ (which is included in the cycle $o$), 
the sublemma is immediate. So we may assume, and we do assume that the path $S_1$ is a subpath of 
the piece $o_i$. As a piece $o_{m}$ with $i<m\leq j$ does not intersect the subpath $L_iM_i$ of the 
piece $o_i$ in interior vertices, the piece $o_{m}$ intersects the path $S_1$ only in the vertex 
$v$. 
It remains to show that the paths $S_1$ and $S_2$ intersect only in the vertex $v$. If one of them 
is a subpath of a path $P_k$, this is immediate. Otherwise, the paths $S_1$ and $S_2$ are proper 
subpaths of the paths $L_iM_i$ and $M_jR_j$ containing the vertex $v$, and thus can only intersect 
in the vertex $v$. 
\end{proof}

We obtain the cycle $o_i'$ from $o_i$ by replacing the path $Q_1$ by $S_1$. 
Similarly, we obtain the cycle $o_j'$ from $o_j$ by replacing the path $Q_2$ by $S_2$.
We obtain the cycle $o'$ from $o$ by replacing the paths $vP_ka_k$ by the paths $S_k$ (for $k=1,2$).

By the choice of the vertex $a_k$ and \autoref{rk_nice}, the cycle $o'$ intersects the 
cycle $o_i'$ precisely in the path $S_1$, and $o'$ intersects the cycle $o_j'$ precisely in the 
cycle $S_2$.

By \autoref{rk_nice}, the cycles $o'$, $o_i'$, $o_j'$ and $o_m$ with $i<m<j$ are the pieces of a 
subdivision of a wheel. 
By construction all these pieces have length at most $r$. Thus this 
subdivision of a wheel is $r$-local. 
\end{proof}

\subsection{Final step}\label{subsec:combine}

\begin{proof}[Proof of \autoref{3con_to_wheel}.]
 Let $G$ be an $r$-locally $3$-connected graph. Our aim is to find an $r$-local subdivision of a 
wheel.  
By \autoref{nice_geo}, the graph $G$ is triangular with parameter $r$ or has 
a geodesic cycle of length at most $r$ with at least four edges. 
If it is triangular, we are done by \autoref{3con_to_wheel_special}.

Hence we may assume, and we do assume, that $G$ has a geodesic cycle of length at most $r$ with 
at least four edges. Denote that cycle by $o'$. Pick two vertices $v_0$ and $v_1$ that are not 
adjacent on the cycle $o'$. Let $P$ be a shortest $v_0$-$v_1$-path included in the geodesic cycle 
$o'$. 
Now apply \autoref{pre-fan-exists} to $o'$, $v_0$ and $v_1$.
We get a pre-fan $F$ of 
parameter $r$ centered at some vertex $v_i$ none of whose pieces contains the vertex $v_{i+1}$ (for 
some $i\in \Fbb_2$). And there is a geodesic cycle $o$ of 
length at most $r$ including $P$ such that the 
start and end of the pre-fan are the neighbours of $v_i$ on $o$. By \autoref{pre-fan_to_fan}, we 
may assume, and we do assume, that the pre-fan $F$ is a fan. 
By symmetry, we may assume, and we do assume, that the center of the fan $F$ is $v_0$. 

As we would be done otherwise, by \autoref{cycle_plus_fan_to_wheel}
there is a piece $o_i$ of the fan $F$ that contains interior vertices of the two $v_0$-$v_1$-paths 
included in the cycle $o$. Indeed, all other assumptions of that lemma are satisfied by $o$ and 
$F$.

\begin{sublem}\label{theta-exists}
 There is a theta-graph $\theta$ of parameter $r$ containing a shortest path between its branching 
vertices with at least two edges. 
\end{sublem}

\begin{proof}
Let $Q$ be a subpath of the cycle $o_i$ that intersects the cycle $o$ precisely at its endvertices 
such that these endvertices are on different $v_0$-$v_1$-paths 
included in the cycle $o$.  

The graph $o\cup Q$ 
is a theta-graph. As the cycle $o$ is geodesic, it includes a shortest path between the branching 
vertices. The cycle obtained by concatenating this path with the path $Q$ cannot be longer than the 
cycle $o_i$. Thus the theta-graph $o\cup Q$ has parameter $r$. 

As the cycle $o$ is geodesic, it contains a shortest path between the two branching 
vertices. Each of the two paths between the branching vertices included in the cycle $o$ contains 
one of the vertices $v_0$ or $v_1$. Hence the theta-graph $o\cup Q$ includes a shortest path with 
at least two edges. 
\end{proof}

Let $\theta$ be a theta-graph as in \autoref{theta-exists}. Denote its two branching vertices by 
$\bar v$ and $\bar w$ and a shortest path between them by $\bar P$. Let $\bar o$ be a cycle of 
$\theta$ including $\bar P$ that does not include an edge between the branching vertices $\bar v$ 
and $\bar w$ (such a choice is possible as $\theta$ contains two paths between the vertices $\bar 
v$ 
and $\bar w$ aside from $\bar P$). 
Now we apply \autoref{pre-fan-exists} to the cycles $\bar o$, the path $\bar P$ and the vertices 
$\bar v$ and $\bar w$.  This part of the proof is similar to the application above but now with the 
theta-graph $\theta$, we have slightly stronger assumptions and will hence be able to complete the 
proof.

 We get a pre-fan $\bar F$ of 
parameter $r$ centered at one of the vertices $\bar v$ or $\bar w$, say $\bar v$, none of 
whose pieces contains the vertex $\bar w$. And there is a cycle 
$\bar u$ of 
length at most $r$ including $\bar P$ such that the 
start and end of the pre-fan are the neighbours of $\bar v$ on $\bar u$. By 
\autoref{pre-fan_to_fan}, we 
may assume, and we do assume, that the pre-fan $\bar F$ is a fan. 

As we would be done otherwise, by \autoref{cycle_plus_fan_to_wheel}
there is a piece $\bar o_i$ of the fan $\bar F$ that contains interior vertices of the two 
$\bar v$-$\bar w$-paths 
included in the cycle $\bar u$. 
By replacing the theta-graph $\theta$ by the theta-graph guaranteed by \autoref{theta_improved} if 
necessary, we may assume, and we do assume, that the cycle $\bar u$ is included in the theta-graph 
$\theta$. By \autoref{theta_short-path} applied to $\theta$ and a suitable subpath of the piece 
$\bar o_i$, the graph $G$ includes an $r$-local subdivision of a wheel.  
\end{proof}

\section{Concluding remarks}

Splitter theorems have turned out to be a key tool when it comes to characterising certain minor 
closed classes of graphs or matroids. And in fact they can be proven in fairly general settings.
For example Chun, Mayhew and Oxley proved a splitter theorem for internally 
4-connected binary matroids \cite{chun2012towards}. We expect that there is the following splitter 
theorem for $r$-locally $3$-connected graphs. 

\begin{con}
Every $r$-locally $3$-connected graphs contains an edge to delete or contract such that $r$-local 
$3$-connectivity is preserved --- unless $G$ is $K_4$.
\end{con}

A \emph{cycle-decomposition} is a graph decomposition whose decomposition graph is a cycle.
For example, every path-decomposition is a cycle decomposition. 
A natural future application of \autoref{intro_series-para} might be to characterise ($r$-locally 
$2$-connected) graphs that have a cycle decomposition of width at most two and locality at least 
$r$.
It is expected that the characterisation of graphs of path-width at most two 
\cite{kinnersley1994obstruction}, see also \cite{yangseries},  extends to this setting in the 
natural way.

More generally, it is expected that \autoref{intro_series-para} can be used to extend excluded 
minors characterisations for any minor-closed subclass of the class of series-parallel graphs 
to our new setting of local separators. 

Another direction in which one could try to extend \autoref{intro_series-para} is to characterise 
graphs with graph-decompositions of width at most three and locality at least $r$. 
The starting point would be the characterisation of graphs of tree-width at most three 
by Arnborg, Proskurowski \& Corneil (1990) and Satyanarayana \& Tung (1990) 
\cite{{arnborg1990forbidden}, {satyanarayana1990characterization}}. 

\bibliographystyle{plain}
\bibliography{literatur}

\end{document}